\documentclass[12pt]{article}

\usepackage{amsmath,amssymb,amsthm,graphicx,fixmath}

\usepackage{algorithmic}
\usepackage{algorithm}
\usepackage{setspace}
\usepackage{ifthen}
\usepackage{subfigure}
\usepackage{tikz}

\def\Re{\mathbb R}

\providecommand{\remove}[1]{}

\theoremstyle{plain}
\newtheorem{theorem}{Theorem}[section]
\newtheorem{lemma}[theorem]{Lemma}

\newtheorem{proposition}[theorem]{Proposition}

\newtheorem{corollary}[theorem]{Corollary}

\theoremstyle{definition}

\theoremstyle{remark}

\newcommand{\A}{\mathcal{A}}
\newcommand{\E}{\mathcal{E}}
\newcommand{\V}{\mathcal{V}}
\newcommand{\B}{\mathcal{B}}
\renewcommand{\L}{\mathcal{L}}

\newcommand{\F}{\mathcal{F}}

\newcommand*{\boundary}{\partial}%

\begin{document}

\title{On the Number of Hyperedges in the Hypergraph of Lines and Pseudo-discs}

\author{Chaya Keller\thanks{Computer Science Department, Ariel University, Ariel, Israel. \texttt{chayak@ariel.ac.il}. Research partially
		supported by Grant 1065/20 from the Israel Science Foundation. Parts of this research were done while the author was at the Mathematics Department, Technion -- Israel Institute of Technology, Israel.}
	\and
	Bal\'azs Keszegh\thanks{Alfr\'ed R{\'e}nyi Institute of Mathematics, and MTA-ELTE Lend\"ulet Combinatorial Geometry Research Group, Institute of Mathematics, E\"otv\"os Lor\'and University, Budapest, Hungary. \texttt{keszegh@renyi.hu}. Research supported by the Lend\"ulet program of the Hungarian Academy of Sciences (MTA), under the grant LP2017-19/2017, by the J\'anos Bolyai Research Scholarship of the Hungarian Academy of Sciences, by the National Research, Development and Innovation Office -- NKFIH under the grant K 132696 and FK 132060 and by the \'UNKP-20-5 New National Excellence Program of the Ministry for Innovation and Technology from the source of the National Research, Development and Innovation Fund.}
	\and
	D\"om\"ot\"or P\'alv\"olgyi\thanks{MTA-ELTE Lend\"ulet Combinatorial Geometry Research Group, Institute of Mathematics, E\"otv\"os Lor\'and University, Budapest, Hungary. \texttt{domotorp@gmail.com}. Research supported by the Lend\"ulet program of the Hungarian Academy of Sciences (MTA), under the grant LP2017-19/2017.
	}
}

\date{}
\maketitle



\begin{abstract}
  Consider a hypergraph whose vertex set is a family of $n$ lines in general position in the plane, and whose hyperedges are induced by intersections with a family of pseudo-discs. We prove that the number of $t$-hyperedges is bounded by $O_t(n^2)$ and that the total number of hyperedges is bounded by $O(n^3)$. Both bounds are tight. 
\end{abstract}

\section{Introduction} 
A family $\F$ of simple Jordan regions in $\Re^2$ is called \emph{a family of pseudo-discs} if for any $c_1,c_2 \in \F$, $|\boundary(c_1) \cap \boundary(c_2)|\leq 2$, where $\boundary(c)$ is the boundary of $c$. Given a set $P$ of points in $\Re^2$ and a family $\F$ of pseudo-discs, define the geometric hypergraph $H(P,\F)$ whose vertices are the points of $P$, and any pseudo-disc $c \in \F$ defines a hyperedge of all points contained in $c$. 

The family of hypergraphs $H(P,\F)$ -- for a general $\F$ and in the special case where all elements of $\F$ are convex -- have been studied extensively (see, e.g., \cite{AKP20,AG86,AU16,CFSS14,MSW90}). In particular, it was proved in \cite{BP13} that for any $P,\F$, the Delaunay graph of $H(P,\F)$ (namely, the restriction of $H$ to hyperedges of size 2) is planar, and that for any fixed $t$, the number of hyperedges of $H(P,\F)$ of size $t$ is bounded by $O(t^2|P|)$. This result was generalized in \cite{K20} (see also~\cite{ADEP21}) to the case where $P$ is a family of pseudo-discs instead of points, and the hyperedges are defined by non-empty intersections of any element in $\F$ with the elements of $P$. 

In this note we consider hypergraphs $H=H(\L,\F)$ whose vertex set $\V(H)=\L$ is a family of lines in the plane, and whose hyperedges are induced by intersections with a family $\F$ of pseudo-discs. Namely, any $c \in \F$ defines the hyperedge
$$e_c=\{ \ell \in \L : \ell \cap c \neq \emptyset    \}  \in \E(H).$$
We assume that the geometric objects are in general position, in the sense that no 3 lines pass through a common point, no line passes through an intersection point of two boundaries of pseudo-discs. 

Unlike the hypergraphs of points w.r.t.~pseudo-discs, $H(P,\F)$, the number of hyperedges in a hypergraph $H(\L,\F)$, of lines w.r.t.~pseudo-discs, of any fixed size, may be quadratic in the number of vertices. 
Such a hypergraph was demonstrated in a beautiful paper of Aronov et al.~\cite{ANPS93}. They showed that for any family $\L$ of lines, if $\F$ consists of the inscribed circles of the triangles formed by any triple of lines, then for any $t \geq 3$, the number of $t$-hyperedges (i.e., hyperedges of size $t$) in $H(\L,\F)$ is exactly ${{n-t+2}\choose{2}}$.

For any fixed $t$, there exist hypergraphs $H(\L,\F)$ in which the number of $t$-hyperedges is larger than in the construction of Aronov et al.~\cite{ANPS93}, even when $\F$ is allowed to contain only discs (as some of those discs might not be inscribed in a triangle formed by the lines). We prove that the number of $t$-hyperedges cannot be significantly larger for any hypergraph $H(\L,\F)$ of lines with respect to pseudo-discs.\footnote{For the difference between hypergraphs induced by pseudo-discs and hypergraphs induced by discs, see~\cite{FS20} and the references therein.} Specifically, we prove:

\begin{theorem}
	\label{thm:main}
	Let $\L$ be a family of $n$ lines in the plane, let $\F$ be a family of pseudo-discs, and assume both families are in general position. Then 
	$$  |\{  e \in \E(H(\L,\F))   :  |e|=t \}|  =O_t(n^2).  $$
\end{theorem}  

Our techniques combine probabilistic and planarity arguments, together with exploiting properties of arrangements of lines, in particular the \emph{zone theorem}.

\medskip In addition, we show that for any choice of $\L$ and $\F$, the total number of hyperedges in $H(\L,\F)$ does not exceed $O(n^3)$. This upper bound is tight, since the total number of hyperedges in the hypergraph presented by Aronov et al.~\cite{ANPS93} is ${{n}\choose{3}}$.
\begin{proposition}
	\label{prop:total}
	Let $\L$ be a family of $n$ lines in the plane, let $\F$ be a family of pseudo-discs, and assume both families are in general position. Then $|\E(H(\L,\F))|=O(n^3)$. 
	
\end{proposition} 

\section{Preliminaries}

In this section we present previous results and simple lemmata that will be used in our proofs.

\subsection{Pseudo-discs}

The two following lemmata are standard useful tools when handling families of pseudo-discs:

\begin{lemma}[Lemma 1 in \cite{Pin14}, based on \cite{SH91}]
	\label{Pin14Lemma1}
	Let $\F$ be a family of pseudo-discs, $D \in \F, x \in D$. Then $D$ can be continuously shrunk to the point $x$, such that at each moment during the shrinking process, the family obtained from $\F$ remains a family of pseudo-discs.
\end{lemma}

\begin{lemma}[Lemma 2 in \cite{Pin14}]
	\label{lem:disjoint}
	Let $\B$ be a family of pairwise disjoint closed connected sets in $\Re^2$. Let $\F$ be a family
	of pseudo-discs. Define a graph $G$ whose vertices correspond to the sets in B and connect two sets
	$B, B' \in \B$ if there is a set $D \in \F$ such that $D$ intersects $B$ and $B'$
	but not any other set from $B$.
	Then $G$ is planar, hence $|E(G)|<3|V(G)|$.
\end{lemma}

\subsection{Arrangements and zones} 

A finite set $\L$ of lines in $\Re^2$ determines an \emph{arrangement} $\A$. The 0-dimensional faces of $\A$ (namely, the intersections of two distinct lines from $\L$), are called \emph{the vertices of $\A$}, the 1-dimensional faces are called \emph{the edges of $\A$}, and the 2-dimensional faces are \emph{the cells of $\A$}. Clearly, all cells are convex. The \emph{cell complexity} of a cell $f$ in $\A$,  denoted by $comp(f)$, is the number of lines incident with the cell. 
The \emph{zone} of an additional line $\ell $, is the set of faces of $\A$ intersected by $\ell$. The \emph{complexity of a zone} is the sum of the cell complexities of the faces in the zone of $\ell$, i.e., total number of edges of these faces, counted with multiplicities.

\begin{theorem}[Zone Theorem \cite{CGL85}]
	In an arrangement of $n$ lines, the complexity of the zone of a line is $O(n)$.
\end{theorem}

The best possible upper bound in the theorem is $\lfloor 9.5(n-1) \rfloor -3$, obtained by Pinchasi~\cite{Pin11}.

\medskip We shall use a generalization of the theorem, for which an extra definition is needed. Given an arrangement $\A$ and a line $\ell$, the 1-zone of $\ell$ is defined as the zone of $\ell$, and for $t>1$ the $t$-zone of $\ell$ is defined as the set of all faces adjacent to the $(t-1)$-zone, that do not belong to any $i$-zone for $i<t$. The \emph{$(\leq t)$-zone} of $\ell$ is the union of the $i$-zones of $\ell$ for all $1 \leq i \leq t$.

The following generalization of the zone theorem was given as Exercise~6.4.2 in~\cite{MATOUSEK}. Its proof can be found in~\cite[Prop.~1]{SSFVCS20}.
\begin{lemma}[\cite{SSFVCS20}]
	\label{lem:zone0}
	Let $\A$ be an arrangement of $n$ lines. Then for any $t$, the $\leq t$-zone of any additional line $\ell $ contains at most $O(tn)$ vertices.
\end{lemma} 
By planarity, this implies:
\begin{corollary}
	\label{lem:zone}
	Let $\A$ be an arrangement of $n$ lines. Then for any $t$, the $\leq t$-zone of any additional line $\ell $ has complexity $C_{\le t}(\ell)=O(tn)$.
\end{corollary}

\subsection{Leveraging from 2-hyperedges to t-hyperedges}

The following lemma allows bounding the number of $t$-hyperedges in a hypergraph $H=(\V,\E)$ in terms of the number of its $2$-hyperedges (i.e., the size of its Delaunay sub-hypergraph) and its \emph{VC-dimension}. 

Let us recall the classical definition of VC-dimension. A subset $\V'\subseteq \V$ is \emph{shattered} if all its subsets are realized by hyperedges, meaning $\{ \V'\cap e\colon e\in \E\} = 2^{\V'}$. The \textit{VC-dimension} of $H$, denoted by $VC(H)$, is the cardinality of a largest shattered subset of $\V$, or $+\infty$ if arbitrarily large subsets are shattered.  

\begin{lemma}[Theorem 6 (ii),(iii) in \cite{AKP21}]\label{Lem:Leverage}
	Let $H=(\V,\E)$ be an $n$-vertex hypergraph. Suppose that there exists an absolute constant $c$ such that for every $\V' \subset \V$, the Delaunay graph of the sub-hypergraph induced by $\V'$ has at most $c|\V'|$ edges. Then the VC-dimension $d$ of $H$ is at most $2c+1$, and the number of hyperedges of size at most $t$ in $H$ is $O(t^{d-1}n)$.
\end{lemma}

The lemma generalizes similar results proved in~\cite{ADEP21,BP13} for hypergraphs of pseudo-discs with respect to pseudo-discs. The assertion regarding the VC-dimension is a simple observation. (Indeed, if a set of $d$ vertices is shattered, then we have ${{d}\choose{2}} \leq cd$, and thus, $d-1 \leq 2c$, or equivalently, $d \leq 2c+1$.) The assertion regarding the number of hyperedges is more involved. 

\section{The number of $t$-hyperedges in $H(\L,\F)$}

In this section we prove Theorem~\ref{thm:main}.
We prove the following stronger statement:
\begin{proposition}\label{Prop:Degree}
	Let $\L$ be a family of $n$ lines in the plane, let $\F$ be a family of pseudo-discs, and assume both families are in general position. Then for each $\ell \in \L$,  
	$$  |\{  e \in \E(H(\L,\F))   :  |e|=t, \ell \in e \}|  =O_t(n).  $$
	Consequently, $|\{  e \in \E(H(\L,\F))   :  |e|=t \}|  =O_t(n^2)$.
\end{proposition}

\begin{proof}[Proof of Proposition~\ref{Prop:Degree}]
	First we prove the statement for hyperedges of size 3, and then we leverage the result to general hyperedges. 
	
	\paragraph{3-hyperedges.} Fix a line $\ell$. We observe that for a pseudo-disc $c$ that defines a 3-hyperedge $  \{  \ell,\ell',\ell''  \} $ there exists a cell of $\A(\L \setminus \{ \ell \} )$ which is in the $\le 2$-zone of $\ell$ in $\A(\L \setminus \{ \ell \} )$ such that $c$ intersects two edges of this cell where one of these edges is on $\ell'$ and the second is on $\ell''$. With every such pseudo-disk $c$ we associate one such cell $f_c$ and one such pair of edges of this cell, and denote this pair by $e_c$. 
	
	Define a graph $G=(V,E)$ whose vertices are all edges in the $(\leq 2)$-zone of $\ell$ in $\A(\L \setminus \{ \ell \} )$, and whose edges are the pairs $e_c$ associated with the pseudo-disks that define a 3-hyperedge. Note that for any hyperedge $e=\{ \ell, \ell', \ell''  \}$ we choose exactly one pair of edges of $\A(\L \setminus \{ \ell \} )$ - one is on $\ell'$ and one is on $\ell''$ - that form a corresponding edge of $G$. Thus by construction, $|E|$ is equal to the number of 3-hyperedges containing $\ell$, and so, we want to prove that $|E|=O(n)$. 
	
	Consider a single cell $f$ of $\A(\L \setminus \{ \ell \} )$. 
	For each pseudo-disk $c$ that defines a 3-hyperedge containing $l$ and has $f_c=f$, $c$ does not intersect any other edge of $f$ besides the two edges in $e_c$ (as otherwise, $c$ would intersect at least 4 lines of $\L$). Hence, the restriction of $G$ to the edges of the cell $f$ (after removing their endpoints), satisfies the assumptions of Lemma~\ref{lem:disjoint}. Thus, by Lemma~\ref{lem:disjoint}, the subgraph of $G$ induced by the edges of $f$ is planar, and hence, its number of edges is at most $3$ times the complexity of $f$. Summing over all cells in the $(\leq 2)$-zone of $\ell$, we obtain $|E| \leq 3\sum_f comp(f)=O(n)$ by Corollary~\ref{lem:zone}, and therefore, $|E|=O(n)$, as asserted.     
	
	
	\paragraph{$t$-hyperedges.} Fix a line $\ell$, and consider the hypergraph $H'$ whose vertex set is $\L \setminus  \{ \ell  \}$ and whose edge set is $\{e \setminus   \{ \ell   \} : e \in \E(H), \ell \in e\}$. The 2-hyperedges of $H'$ correspond to 3-hyperedges of $H$ containing $\ell$, and thus, by the first step, their number is $O(n)$. Furthermore, for any $\L' \subset \L \setminus  \{ \ell  \}$, the number of 2-hyperedges in the restriction of $H'$ to $\L'$ is $O(|\L'|)$, by the same argument. Therefore, $H'$ satisfies the assumptions of Lemma~\ref{Lem:Leverage}, which implies that the VC-dimension $d$ of $H'$ is constant, and that the number $C_{t-1}$ of $(t-1)$-hyperedges of $H'$ is $O(t^{d-1}n)$. 

	Finally, the number of $t$-hyperedges of $H$ that contain $\ell$ is equal to $C_{t-1}$. This completes the proof.
\end{proof}

\section{The total number of hyperedges in $H(\L,\F)$}

In this section we prove Proposition \ref{prop:total}. 

\begin{proof}[Proof of Proposition~\ref{prop:total}]
	By Lemma~\ref{Pin14Lemma1} we can shrink the pseudo-discs one by one, such that the shrinking of each pseudo-disc $c \in \F$ is stopped when it becomes tangent to two lines. (Formally, first $c$ is shrunk until the first time it is tangent to some line in $\L$, and then it is shrunk towards the tangency point until the next time it is tangent to some line in $\L$.) By the general position assumption, we can perform the shrinking process in such a way that the obtained geometric objects (i.e., lines and shrinked pseudo-discs) are also in general position.  We replace each $c \in \F$ by its shrunk copy. Let $\F'$ be the obtained family. Then $H(\L,\F)= H(\L,\F')$, and by a tiny perturbation we can assume that all tangencies are in a point. 
	
	For any two lines $\ell_1,\ell_2 \in \L$, denote by $\F'(\ell_1,\ell_2)$ the set of all pseudo-discs in $\F'$ that are tangent to both $\ell_1$ and $\ell_2$. We claim that for any $\ell_1,\ell_2 \in \L$, $|   \E(H(   \L,\F'  (\ell_1,\ell_2)   ))           |=O(n)$, and this implies $|\E(H)|=O(n^3)$, the assertion of Proposition \ref{prop:total}.
	
	To show this, for any $c \in \F'  (\ell_1,\ell_2)$, we define $x_{\ell_1,\ell_2}(c)=c \cap \ell_1 \in \Re^2$ and $y_{\ell_1,\ell_2}(c)=c \cap \ell_2 \in \Re^2$ (see Figure \ref{fig5}). In each of the four wedges that $\ell_1,\ell_2$ form, we define a linear order relation on the elements of $\F'  (\ell_1,\ell_2)$: $c \prec  c'$ if the segment $[x_{\ell_1,\ell_2}(c),y_{\ell_1,\ell_2}(c)]$ is completely above the segment $[x_{\ell_1,\ell_2}(c'),y_{\ell_1,\ell_2}(c')]$ (that is, if the points $x_{\ell_1,\ell_2}(c),y_{\ell_1,\ell_2}(c)$ are closer to the intersection point within the wedge than the points $x_{\ell_1,\ell_2}(c'),y_{\ell_1,\ell_2}(c')$, respectively). 
	
	First, we claim that this relation is well defined, since 
	for $c \neq c'$ two such segments never intersect. Indeed, 
	assume to the contrary they intersect, so that $y_{\ell_1,\ell_2}(c')$ is above $y_{\ell_1,\ell_2}(c)$, while $x_{\ell_1,\ell_2}(c')$ is below $x_{\ell_1,\ell_2}(c)$. The pseudo-disc $c$ divides the remainder of the wedge into two connected components -- the part `above' it and the part `below' it. Now, consider the points $x_{\ell_1,\ell_2}(c'), y_{\ell_1,\ell_2}(c')$. In the boundary of $c'$, these points are connected by two curves. As these points are in different connected components w.r.t.~$c$, each of these curves intersects $c$ at least twice, which means that $c,c'$ intersect at least 4 times, a contradiction. 
	
	Second, we claim that in each wedge, every line in $\L$ intersects a subset of consecutive elements of $\F'  (\ell_1,\ell_2)$ under the order $\prec$. Indeed, assume that some line $\ell$ intersects two pseudo-discs $c_1,c_3$, as depicted in Figure~\ref{fig5}. We want to show it must intersect $c_2$ as well. Like above, $c_2$ divides the wedge (without it) into two connected components. By the same argument as above, $c_1$ cannot intersect the component below $c_2$ (as otherwise, it would cross $c_2$ four times). Similarly, $c_3$ cannot intersect the component above $c_2$. Thus, either $\ell$ intersects at least one of $c_1,c_3$ inside $c_2$, or $\ell$ contains a point above $c_2$ and a point below $c_2$. In both cases, $\ell$ must intersect $c_2$.
	
	Finally, by passing over all elements of $\F'  (\ell_1,\ell_2)$ in each wedge, from the smallest to the largest, according to the order $\prec$, the number of times that the hyperedge defined by the current pseudo-disc is changed is linear in $|\L|$. Indeed, any such change is caused by appearance or disappearance of some line, and each line in $\L$ appears at most once and disappears at most once, along the proccess. Therefore, in each wedge, $|    \E(H(  \L, \F'(\ell_1,\ell_2  )    ))         | =O(n)$, and summing over all pairs $\{ \ell_1,\ell_2 \} \in \L$, we get $|\E(H)|=O(n^3)$.
\end{proof}

\begin{figure}[tb]
	\begin{center}
		\scalebox{0.6}{
			\includegraphics{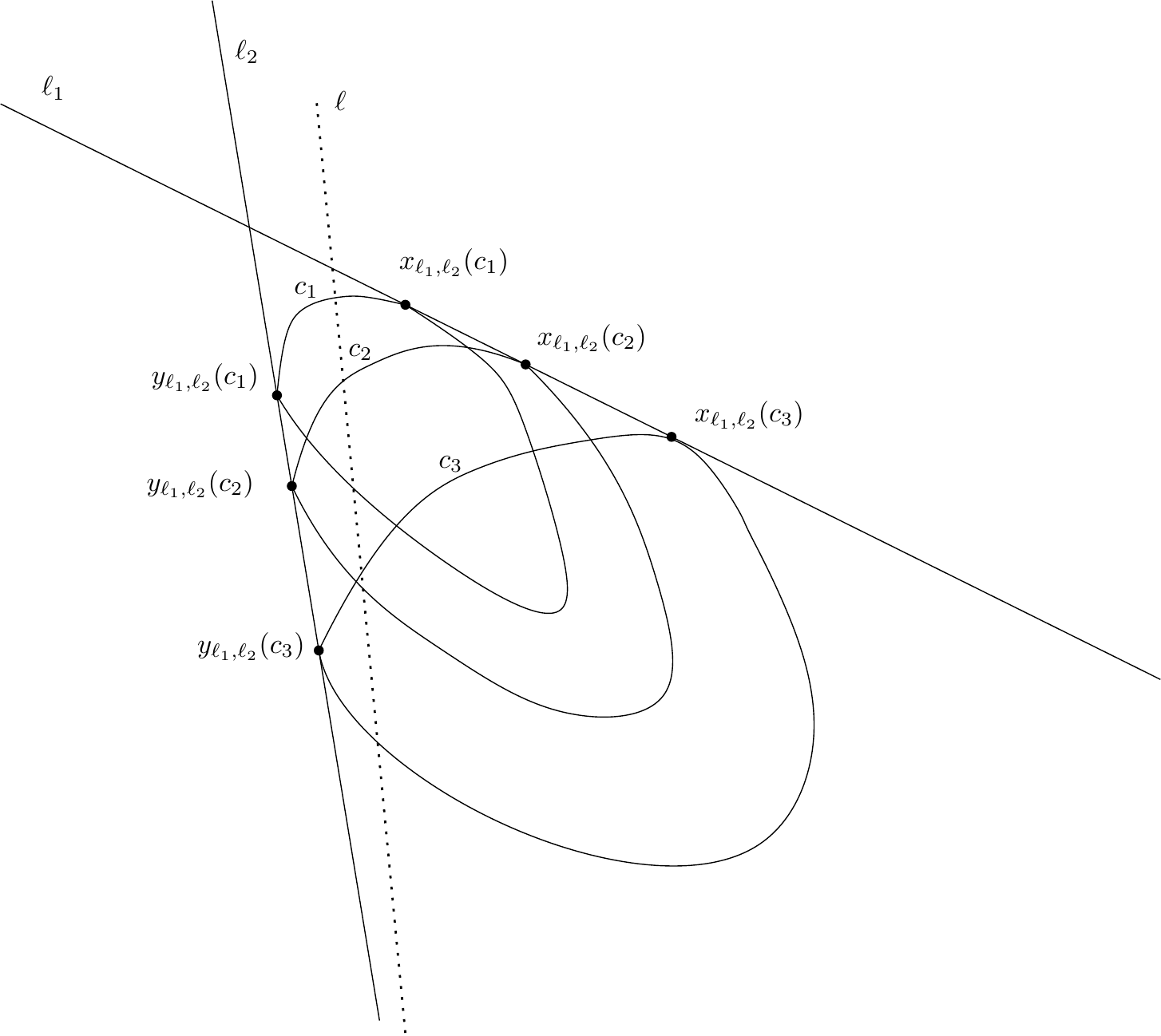}
		}
		\caption{Illustration for the proof of Proposition \ref{prop:total} - $c_1,c_2,c_3$ are tangent to the lines $\ell_1,\ell_2$, and $c_1\prec c_2 \prec c_3$.}
		\label{fig5}
	\end{center}
\end{figure}

\section{Open Problems}

We conclude this note with a few open problems.

\paragraph{Hypergraph of lines and inscribed pseudo-discs.} A natural question is whether the arguments of Aronov et al.~\cite{ANPS93} can be extended from discs to pseudo-discs.
We have found that all their arguments would go through if we knew that every triangle has an inscribed pseudo-disc.
More precisely, we would need that for any triangle formed by three sides $a,b,c$, there is a pseudo-disc $d\in \F$, contained in the closed triangle, that intersects every side in exactly one point, or if there is no such $d\in\F$, then we can add such a new pseudo-disc $d$ to $\F$ such that $\F\cup\{d\}$ still forms a pseudo-disc family.
Unfortunately, it seems that such a theory has not been developed yet, not even for $\F$ all whose elements are convex.

We note that for the related problem regarding circumscribed pseudo-discs, even a stronger result is known. Specifically, it was shown in~\cite[Thm.~5.1]{SH91} that for any three points $a,b,c$, there is a pseudo-disc $d\in \F$ such that $a,b,c\in\partial d$, or if there is no such $d\in\F$, then we can add such a new pseudo-disc $d$ to $\F$ such that $\F\cup\{d\}$ still forms a pseudo-disc family. 

\paragraph{Dependence on $t$ in Theorem~\ref{thm:main}.} While we showed the quadratic dependence on $n$ in Theorem~\ref{thm:main} to be tight, the dependence on $t$ is not clear. It seems plausible that 
$$  |\{  e \in \E(H(\L,\F))   :  |e|=t \}|  =O(t n^2),  $$
but we have not been able to prove this. On the other hand, even the stronger upper bound $O(n^2)$ for any fixed $t$, that would immediately imply Proposition~\ref{prop:total} might hold.

\paragraph{Analogue of Lemma~\ref{Lem:Leverage} for 3-sized hyperedges.} It seems plausible that one can prove the following analogue of Lemma~\ref{Lem:Leverage} for 3-sized hyperedges: If in some hypergraph on $n$ vertices, for any induced hypergraph, the number of 3-sized hyperedges is quadratic in the number of vertices, then for any fixed $t$, the number of $t$-sized hyperedges is $O_t(n^2)$. Such a strong leveraging lemma would allow an easier proof of Theorem~\ref{thm:main}.

\section*{Acknowledgements}

The authors are grateful to Rom Pinchasi for inspiring and helpful suggestions, to Stefan Felsner for suggesting to use Lemma~\ref{lem:zone0} and for other valuable suggestions, and to Manfred Scheucher for useful discussions on arrangements of pseudo-discs.



\end{document}